\documentclass[12pt]{amsart}
\usepackage[english]{babel}
\usepackage[margin=1in]{geometry}

\usepackage{graphicx,amssymb,latexsym,amsmath,amscd,amsthm,amsfonts}
\usepackage{caption}
\usepackage{subcaption}
\usepackage{tabularx}
\usepackage{verbatim, color}

\usepackage{epstopdf}

\ifpdf
  \DeclareGraphicsExtensions{.eps,.pdf,.png,.jpg}
\else
  \DeclareGraphicsExtensions{.eps}
\fi
\usepackage{verbatim, color} 

\newcommand{\ra}{\rightarrow}

\theoremstyle{plain}% Theorem-like structures
\newtheorem{thm}{Theorem}[section]

\newtheorem{lemma}[thm]{Lemma}
\newtheorem{prop}[thm]{Proposition}

\theoremstyle{definition}

\newtheorem{exam}[thm]{Example}

\theoremstyle{remark}
\newtheorem{rem}[thm]{Remark}
\newtheorem{rems}[thm]{Remarks}

\newcommand{\E}{\mathcal{E}}
\newcommand{\G}{\mathcal{G}}

\begin{document}

\title{A Simple Loop Dwell Time Approach for Stability of Switched Systems} 
\author{Nikita Agarwal}
\address{Department of Mathematics, Indian Institute of Science Education and Research Bhopal, Bhopal Bypass Road, Bhauri, Bhopal 462 066, Madhya Pradesh, INDIA}
\email{nagarwal@iiserb.ac.in}

\begin{abstract}
We introduce a novel concept of simple loop dwell time and use it to give sufficient conditions for stability of a continuous-time linear switched system where switching between subsystems is governed by an underlying graph. We present a slow-fast switching mechanism to ensure stability of the system. We also consider switched systems with both stable and unstable subsystems, and obtain bounds on the dwell time in the stable subsystem and flee time from the unstable subsystem that guarantee the stability of the system. 
\end{abstract}

% REQUIRED
%\begin{keywords}
%piecewise-continuous dynamical systems, graph theory, control theory 
%\end{keywords}
%
%% REQUIRED
%%\begin{AMS}
%%37N35, 68R05, 93D20
%%\end{AMS}
\maketitle
\section{Introduction} \label{intro} 
A continuous-time switched system is a piecewise continuous dynamical system with finitely many subsystems, and a piecewise constant function, known as the switching signal, which determines the switching of the system between subsystems. A signal is represented by the admissible switching from one subsystem to another, and the times at which these switchings take place. In this study, the switching between subsystems will be governed by an underlying digraph. That is, the system can switch from a subsystem to another if there is a directed edge between the corresponding vertices on the underying graph. Such systems have been studied in~\cite{IK, O1, KIO, KS, Man}. Switched systems have applications in electrical and power grid systems, where the underlying graph structure varies with time. A review on switched systems as an evolving dynamical systems, and its potential applications is presented in~\cite{B1} and references therein. The synchronization of time-varying networks is addressed in~\cite{St} using the concept of averaged topology, and in~\cite{B4} using a method called connection graph stability method. Synchronization of time-varying topologies due to moving agents is considered in~\cite{SB} . In~\cite{B3, B2, PSB, PSBS}, networks with randomly changing topologies are studied. It was observed that strongly connected components of graphs play an important role in understanding the network. In~\cite{Man}, the stability conditions for switched systems are reduced to conditions on strongly connected components of the graph. 

Even when all the subsystems
are stable, the switched system may be unstable for some switching signal. Moreover, one can construct a signal which can stabilize a switched system with all unstable subsystems. Thus, it is evident that the stability of a switched system not only depends on the properties of subsystems, but also on the switching signal. In this paper, we will give sufficient conditions on the switching times under which the switched system will be stable. In~\cite{O1}, a lower bound on
dwell time and average dwell time is obtained for the stability of such systems using the maximum cycle ratio and the maximum cycle mean of the associated switching graph. The concepts of dwell time and average dwell time were introduced in~\cite{Morse1} and~\cite{HM}, respectively. In this paper, for a switched system with all stable subsystems, we will obtain lower bounds on the simple loop dwell time, which is the minimum total time that the signal spends on each simple loop, that guarantees the stability of the system. To formalize the notion of total time spent on a simple loop, we introduce a standard decomposition algorithm. This approach has an advantage over the dwell time and average dwell time approaches, since the signal can switch slowly on some edges on a simple loop, and faster on some of its other edges. This gives rise to signals with a combination of slow and fast switching. Hence, a switched system, which would otherwise seem to be unstable, can be stabilized using the concept of simple loop dwell time. A similar phenomenon was observed in~\cite{GJBP}, where emergence of windows of opportunity for synchronization is exhibited numerically in coupled stochastic maps. The windows of opportunity for stability in continuous-time stochastic communication network was studied in~\cite{JB}. 

Further, when the switched system comprises of both stable and unstable subsystems, we give sufficient stability conditions on the switched signal. In addition to the notion of dwell time in a stable subsystem, we introduce the notion of flee time, which is the maximum time that the signal spends in the unstable subsystem. We obtain a lower bound on the dwell time, and an upper bound on the flee time, which ensures stability of the switched system. We also give a slow-fast mechanism to promote stability, as done for switched systems with stable subsystems only. Further, under a hypothesis on the underlying graph, we give bounds for the dwell time and the flee time, which ensures stability of the system. This uses the concept of topological sorting for acyclic graphs. Stability of switched systems with both stable and unstable subsystems have been considered in~\cite{SU}, using average dwell time approach. In~\cite{HXMA}, stability results are given for the case when all the subsystem matrices commute pairwise. No such condition on subsystem matrices is assumed here.

The paper is organized as follows: in~\ref{background}, we give some necessary background material on graphs, a graph-dependent switched system, and notion of its stability. In~\ref{stable}, we consider the stability of graph-dependent switched systems with all stable subsystems by finding suitable bounds on the simple loop dwell time with respect to the standard decomposition, given in Section~\ref{stdec}. The stability results for such systems are presented in Section~\ref{class}. The stability of graph-dependent switched systems with both stable and unstable subsystems is considered in~\ref{stable-unstable}. The results for special cases of a unidirectional ring and a bipartite graph are given in Sections~\ref{uni} and~\ref{bip}, respectively. Switched system associated to an arbitrary graph is given in Section~\ref{gen}.

\section{Background}\label{background}
In this section, we give some preliminaries on digraphs and describe a continuous-time switched system whose switching is given by an (infinite) path on an underlying graph. We let $\mathbb{N}$ denote the natural numbers. If \mbox{$k\in\mathbb{N}$}, 
we use the notation $\mathbf{k}=\{1,\cdots,k\}$. For a matrix $M$, $\Vert M\Vert$ will denote its spectral norm.

\subsection{Graphs}
A directed graph (or a digraph) is a set of vertices and directed edges from one vertex to another. In this paper, we assume that there is atmost one edge from one vertex to another. For simplicity of notation, we label the vertices of a graph $\mathcal{G}$ with $k$ vertices by $v_1,\cdots,v_k$. The vertex set $\{v_1,\cdots,v_k\}$ is denoted by $v(\G)$. Associated to every such graph $\mathcal{G}$, the edge set $\mathcal{E}(\mathcal{G})$ is the collection of all tuples $(i,j)$, where there is an edge from vertex $v_i$ to $v_j$, for $i,j\in\mathbf{k}$. The \emph{adjacency matrix} of the graph $\mathcal{G}$ is a $k\times k$ matrix $A_{\G}=[a_{ij}]$ given by $a_{ij}=1$, if there is an edge from $v_i$ to $v_j$. If there is no edge from $v_i$ to $v_j$, then $a_{ij}=0$. 
For $i\in\mathbf{k}$, the \emph{indegree} of the vertex $v_j$ is the $j^{\rm th}$ column sum of $A_\G$ 
and is the total number of incoming edges to the vertex $v_j$.  
Similarly, for $i\in\mathbf{k}$, the \emph{outdegree} of the vertex $v_i$ is the $i^{\rm th}$ row sum of $A_\G$ 
and is the total number of outgoing edges from the vertex $v_i$. A \emph{path} in the graph $\mathcal{G}$ is a sequence of 
vertices and directed edges such that from each vertex there is an edge to the next vertex in the sequence. The number of edges describing a path $p$ is called the \emph{length of the path}, denoted by $\ell(p)$. A path will be denoted either by the sequence of labels of vertices, or the sequence of edges, in the order they appear on the path. For two paths $p_1$ and $p_2$ in $\G$, their union $p=p_1+p_2$ denotes the path with $v(p)=v(p_1)\cup v(p_2)$, and the edge set of $p$ is the union of edge sets of $p_1$ and $p_2$, counting multiplicity. A \emph{loop} is a closed path; that is, a path whose terminal vertices are the same. An \emph{acyclic graph} is a graph without any loops. A loop is called a \emph{simple loop} if all the vertices on that loop are distinct. It is easy to see that every loop can be uniquely expressed as a union of simple loops. A graph is said to be \emph{strongly connected} 
if there is a path from each vertex to every other vertex. 

\begin{rem}
The maximum number of simple loops in a directed graph with $k$ vertices with adjacency matrix $A_\G$ is $\sum_{r=1}^k \sum_{i=1}^{k} (A_\G^r)_{ii}$. There are several algorithms to find all the simple loops in the graph $\G$.
\end{rem}

\subsection{Graph-dependent switched system}
Let $\G$ be a digraph with $k$ vertices $\{v_1,\cdots,v_k\}$. Let $\sigma:[0,\infty)\ra \{1,\cdots,k\}$ be a right-continuous piecewise constant function taking values in $\{1,\cdots,k\}$ with discontinuities $0=t_0<t_1<t_2<\cdots$, such that $(\sigma(t_i),\sigma(t_{i+1}))\in\mathcal{E}(\G)$, for all $i\geq 0$. Let $\sigma_i$ denote the value of $\sigma$ in the time interval
$[t_{i-1},t_i)$, for $n\geq 1$. Thus $\sigma_1,\cdots,\sigma_n$ is a path of length $n$ in $\G$. Such a signal $\sigma$ is called a $\G$-admissible signal. Each $\G$-admissible signal is identified by the following data: switching times $(t_n)_{n\geq 1}$, an increasing sequence of positive real numbers, and an infinite path $(\sigma_n)_{n\geq 1}$ in $\G$ (that is, $(\sigma(t_{n-1}),\sigma(t_{n}))\in\mathcal{E}(\G)$, for all $n\geq 1$). Let $\mathcal{S}_\G$ denote the collection of all $\G$-admissible signals. \\
Let $A_1,\cdots,A_k$ be $n\times n$ matrices with real entries. We call a matrix \textit{stable} if all its eigenvalues have negative real part, and a matrix is called \textit{unstable} if it has atleast one eigenvalue with positive real part, and no eigenvalue with zero real part.\\\\ 
For $\sigma\in \mathcal{S}_\G$, consider the switched linear system in $\mathbb{R}^n$ given by 
\begin{equation}\label{main} 
	x^\prime(t) = A_{\sigma(t)}x(t), \ \ t\geq 0.
\end{equation} 
The system~\ref{main} is called \textit{a switched system with a $\G$-admissible signal $\sigma\in \mathcal{S}_\G$}.\\
For each $i\geq 1$, the linear system $x^\prime(t) = A_{\sigma_i}x(t)$, $t\in [t_{i-1},t_i)$, 
is called a subsystem of~\ref{main}. A subsystem is known as \textit{stable} if $A_{\sigma_i}$ is a stable matrix. If $A_{\sigma_i}$ is an unstable matrix, the subsystem is known as \textit{unstable}. Throughout this article, we will assume that for each $j\in\mathbf{k}$, $A_j$ is a diagonalizable (over $\mathbb{C}$) matrix, see~\ref{diag} about the diagonalizability hypothesis. We consider the real Jordan form $A_j=P_jD_jP_j^{-1}$, where the columns of $P_j$ are the eigenvectors of $A_j$ with unit norm. The matrices $A_1,\cdots,A_k$ are called subsystem matrices of the switched system.

\begin{rem}
\label{zeno}
If the switching times have an accumulation point, we say that the system exhibits zeno behavior. Examples of such a behavior are given in~\cite[Section 1.2.2]{Lib}. Observe that if the sequence $(t_n)$ is infinite and is bounded above, then it has an accumulation point. In this article, we will assume that the zeno behavior does not occur, and $t_n\ra\infty$, as $n\ra\infty$. 
\end{rem}

\begin{rem}\label{diag}
If $A=PDP^{-1}$ is a diagonalizable matrix, then $\Vert e^{Ds}\Vert \leq e^{\lambda s}$, where \\
$\lambda=\max\{\textit{real part of eigenvalues of}\ A\}$. If $A=PDP^{-1}$ is not diagonalizable, then for each $\lambda^*>\max\{\textit{real part of eigenvalues of}\ A\}$, there exists $\beta>0$ such that $\Vert e^{Ds}\Vert \leq \beta e^{\lambda^* s}$. All the estimates obtained in this paper will include $\lambda^*$ and $\beta$ corresponding to each $A_i$, when the matrices are non-diagonalizable.
\end{rem}

\begin{exam} 
Consider a uni-directional cycle $\mathcal{G}$ with $k$ vertices. That is, $\mathcal{E}(\mathcal{G})=\{(i,i+1), (k,1)\ \vert\ i=1,\cdots,k-1\}$. Thus if $\sigma_n=r(<k)$, then the only choice for $\sigma_{n+1}$ is $r+1$, and if $\sigma_n=k$, then the only choice for $\sigma_{n+1}$ is $1$. Hence any $\G$-admissible signal $\sigma\in\mathcal{S}(\G)$ satisfies $(\sigma_n)_{n\geq 1}=\overline{r(r+1)\cdots k1\cdots (r-1)}$, for $r\in\{1,\cdots,k\}$.
\end{exam}  

\subsection{Stability of a switched system}
A graph-dependent switched system~(\ref{main}) with $\sigma\in\mathcal{S}_{\G}$ is \textit{asymptotically stable} if for all initial conditions $x(0)\in \mathbb{R}^n$, $\lim_{t\ra \infty}\Vert x(t)\Vert= 0$. \\
In this article, for a given digraph $\G$, we will consider the problem of characterizing $\G$-admissible signals for which the switched system given by~\ref{main} is asymptotically stable. Since we are restricting ourselves to linear systems, and the eigenvalues of $A_1,\cdots,A_k$ are away from the imaginary axis, asymptotic stability is the only kind of stability which is possible. In~\ref{stable}, we consider switched systems in which all the subsystems are stable and in~\ref{stable-unstable}, the switched systems have both stable and unstable subsystems.

\section{Switched system with all stable subsystems}
\label{stable}
Let $\mathcal{G}$ be a directed graph with $k$ vertices $v_1,\cdots, v_k$. Consider the switched system~\ref{main} with $\sigma\in \mathcal{S}_\G$. \\
In this section, we will assume that each $A_j$ is Hurwitz, that is, each subsystem of~\ref{main} is stable. It is known that there may exist signals $\sigma$ (with all-to-all connected underlying graph) for which the switched system~\ref{main} is not stable, we refer to~\cite{Lib} for examples. It is also known that if the time interval between consecutive switches is bounded below by a sufficiently large quantity (known as the dwell time), then the switched system is stable, see for example~\cite{O1}. 
\noindent For each $i\in\bf{k}$, let $-\lambda_i$ be the maximum of the real part of eigenvalues of $A_i$. Note that the eigenvalue(s) of $A_i$ with real part $-\lambda_i$ is the one closest to the imaginary axis.\\
\noindent For $t\in [t_{n-1},t_n)$, the solution of the switched system~\ref{main} with initial condition $x(0)$ is given by $x(t)=
e^{A_{\sigma_n}(t-t_{n-1})}e^{A_{\sigma_{n-1}}(t_{n-1}-t_{n-2})}\cdots e^{A_{\sigma_{1}}t_{1}} x(0)$.

\noindent Thus we have
\begin{equation} \label{main2'}
\begin{split}
\Vert x(t)\Vert &= \Vert e^{A_{\sigma_n}(t-t_{n-1})}e^{A_{\sigma_{n-1}}(t_{n-1}-t_{n-2})}\cdots e^{A_{\sigma_{1}}t_{1}} x(0)\Vert   \\ 
&= \Vert P_{\sigma_n}e^{D_{\sigma_n}(t-t_{n-1})}P_{\sigma_n}^{-1}\left(\prod_{j=1}^{n-1}P_{\sigma_{n-j}}e^{D_{\sigma_{n-j}}(t_{n-j}-t_{n-j-1})}P_{\sigma_{n-j}}^{-1}\right) x(0)\Vert  \\
&\leq  \Vert P_{\sigma_n}\Vert \Vert P_{\sigma_1}^{-1}\Vert e^{-\lambda_{\sigma_n}(t-t_{n-1})}\left(\prod_{j=1}^{n-1}\Vert
P_{\sigma_{j+1}}^{-1}P_{\sigma_{j}}\Vert e^{-\lambda_{\sigma_j}(t_j-t_{j-1})}\right)\Vert x(0)\Vert \\ 
&\leq \rho e^{-\lambda_{\sigma_n}(t-t_{n-1})} \left(\prod_{j=1}^{n-1}\Vert P_{\sigma_{j+1}}^{-1}P_{\sigma_{j}}\Vert e^{-\lambda_{\sigma_j}(t_j-t_{j-1})}\right)\Vert x(0)\Vert, 
\end{split}
\end{equation} 
where $\rho = \max\{\Vert P_{j}^{-1}\Vert \Vert P_{i}\Vert\ \vert\ \text{there is a path in $\G$ from $v_i$ to $v_j$},\ i, j\in\mathbf{k}\}$, 
which depends on the graph $\mathcal{G}$, but is independent of the signal $\sigma$. 

\begin{rems}\label{rem1}
1) If $\mathcal{G}$ is strongly connected then $\rho=\max_{i,j\in\mathbf{k}}\Vert P_{j}^{-1}\Vert \Vert P_{i}\Vert$.\\
2) If $\G$ has no loops, then every path in $\G$ has length atmost $k=\vert \G\vert$. Hence, any switching signal is eventually constant. Thus every switched system with a $\G$-admissible signal is stable. Moreover, if the graph $\mathcal{G}$ has a vertex $v_\ell$ with zero outdegree, and a signal $\sigma$ assumes the value $\ell$, then the switched system is stable. Thus, we will restrict our attention to graphs in which each vertex has non-zero outdegree. It should be noted that such graphs have atleast one simple loop since the number of vertices is finite. \\
%Furthermore, such a graph need not be strongly connected. For example, $\mathcal{G}$ with three vertices with $\mathcal{E}(\mathcal{G})=\{(1,1),(2,3),(3,2)\}$ has each vertex with non-zero outdegree, but is not strongly connected.
3) If $\rho<1$, the last inequality in~\ref{main2'} gives $\Vert x(t)\Vert\leq \rho^n \Vert x(0)\Vert$, for all $t\in [t_{n-1},t_n)$. Hence the switched system is stable.\\
4) If $\G$ has a loop, then $\rho\geq 1$, since for any invertible matrices $A, B$, $\Vert A\Vert\Vert B^{-1}\Vert \Vert B\Vert\Vert A^{-1}\Vert\geq 1$. 
\end{rems}

\noindent In view of the above remarks, we assume the following hypothesis:\\
\noindent \textbf{(H1)} The underlying graph $\G$ has a loop, and for given $T>0$, $\sigma$ is discontinuous at some $t\geq T$.\\

\noindent Let $\mathcal{G}$ have $p$ simple loops, $s_1,\cdots,s_p$. For $t\in [t_{n-1},t_n)$, the last inequality in~\ref{main2'} gives $\Vert x(t)\Vert \leq a_n^\sigma \rho \Vert x(0)\Vert$, where 
\begin{eqnarray}\label{ansigma}
a_n^\sigma &=& \prod_{j=1}^{n-1}\Vert
	P_{\sigma_{j+1}}^{-1}P_{\sigma_{j}}\Vert e^{-\lambda_{\sigma_j}(t_j-t_{j-1})}. 
\end{eqnarray}
	
\subsection{Standard Decomposition Algorithm}\label{stdec}
For a given graph $\G$ with vertices $\{v_1,\cdots,v_k\}$, consider a $\G$-admissible switching signal $\sigma\in \mathcal{S}(\G)$, with associated switching times $(t_n)_{n\geq 1}$ and an infinite path $(\sigma_n)_{n\geq 1}$ in $\G$, with edges $e_n=(v_{\sigma_n},v_{\sigma_{n+1}})$, $n\geq 1$. To each edge $e_n$, we associate the time $\delta_n=t_{n+1}-t_n$, which is the time that the signal spends in the $\sigma_n^{th}$ subsystem before it switches to the $\sigma_{n+1}^{th}$ subsystem. The standard decomposition algorithm of $\sigma^{(n)}=\sigma_1\sigma_2\cdots\sigma_n$ is as follows: \\
\noindent \textbf{Step 1}: Let $P_0=\sigma_1\sigma_2\cdots\sigma_n$ with edges $e_1, e_2,\cdots, e_{n-1}$, and let $i(P_0)$ denote the set consisting of subscripts $j$ of all $e_j$ that appear in $P_0$. Let $k_2\in i(P_0)$ be the minimum index such that $\sigma_{k_2}=\sigma_{j+1}$ for some $j<k_2$ in the index set $i(P_0)$. Let $k_1\in i(P_0)$ be such that $k_1<k_2$ and $\sigma_{k_1}=\sigma_{k_2+1}$. If such a pair does not exist, then the path $P_0$ is indecomposable and the algorithm stops. Otherwise, we proceed to Step 2. It is easy to see that the subpath $P^0=\sigma_{k_1}\sigma_{k_1+1}\cdots \sigma_{k_2}$ with edges $e_{k_1},\cdots,e_{k_2-1}$ of $P_0$ is a simple loop in $\G$. The total time spent by $\sigma$ on the simple loop $P^0$ is given by $\delta_{k_1}+\cdots+\delta_{k_2-1}$.\\
\noindent \textbf{Step 2}: Let $P_1=P_0\setminus P^0$ be the path obtained by deleting the edges of $P^0$ from $P_0$. If $P_1$ is indecomposable, the algorithm stops, otherwise repeat Step 1 by replacing $P_0$ by $P_1$. \\
Using this algorithm, $\sigma^{(n)}$ can be decomposed into simple loops and an indecomposable path. Such a decomposition is called the \textit{standard decomposition}.  
\begin{figure}[h!]
\centering
\includegraphics[width=.5\textwidth]{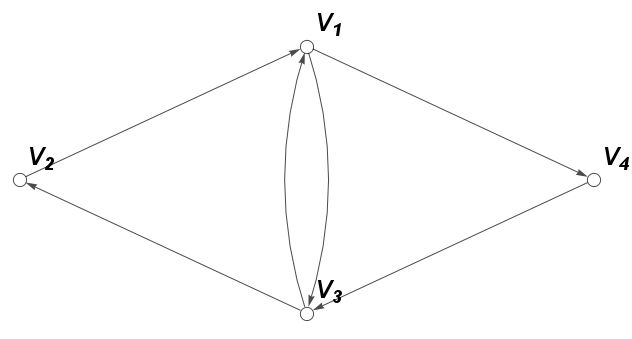}
\caption{The underlying graph $\G$ in~\ref{standard}}
\label{adj_exam11}
\end{figure}

\begin{exam}\label{standard}
In this example, we will illustrate the standard decomposition algorithm. Let $\G$ be the graph given in~\ref{adj_exam11}. There are four simple loops in $\G$, namely $s_1=\{v_1\ra v_3\ra v_2\ra v_1\}$, $s_2=\{v_1\ra v_4\ra v_3\ra v_1\}$, $s_3=\{v_1\ra v_4\ra v_3\ra v_2\ra v_1\}$, and $s_4=\{v_1\ra v_3\ra v_1\}$. Consider a $\G$-admissible signal $\sigma\in\mathcal{S}_\G$ with $\sigma^{(13)}=(\sigma_n)_{1\leq n\leq 13}=(2,1,3,1,4,3,2,1,4,3,1,4,3)$. \\
As shown in Table~\ref{table}, the standard decomposition of $\sigma^{(13)}$ is obtained in three steps (a-c), and is given by three simple loops $e_2e_3=s_4$, $e_1e_4e_5e_6=s_3$, $e_8e_9e_{10}=s_2$, and an indecomposable path $e_7e_{11}e_{12}$. The total time spent by $\sigma$ on the simple loops in the standard decomposition of $\sigma^{(13)}$ is given by $\delta_2+\delta_3$ on $e_2e_3=s_4$, $\delta_1+\delta_4+\delta_5+\delta_6$ on $e_1e_4e_5e_6=s_3$, and $\delta_8+\delta_9+\delta_{10}$ on $e_8e_9e_{10}=s_2$.
\end{exam}

\begin{table} 
\centering
\caption{Standard Decomposition Algorithm applied to $\sigma^{(14)}$.}
\begin{tabular}{|c c c c c c c c c c c c c|}
 \hline
$\sigma_1$ & $\sigma_2$ & $\sigma_3$ & $\sigma_4$ &$\sigma_5$ & $\sigma_6$ &$\sigma_7$ & $\sigma_8$ &$\sigma_9$ & $\sigma_{10}$ &$\sigma_{11}$ & $\sigma_{12}$ &$\sigma_{13}$  \\
 \hline
2 & 1 & 3 & 1 & 4 & 3 & 2 & 1 & 4 & 3 & 1 & 4 & 3 \\
\hline
\end{tabular}

\begin{tabular}{|c c c c c c c c c c c c|}
 \hline
$e_1$ & $e_2$ & $e_3$ & $e_4$ &$e_5$ & $e_6$ &$e_7$ & $e_8$ &$e_9$ & $e_{10}$ &$e_{11}$ & $e_{12}$   \\
 \hline
(2,1) & (1,3) & (3,1) & (1,4) & (4,3) & (3,2) & (2,1) & (1,4) & (4,3) & (3,1) & (1,4) & (4,3) \\
 \hline
\end{tabular}
\bigskip
\subcaption*{(a) After removing edges $e_2$ and $e_3$ corresponding to $s_4$ from $\sigma^{(14)}$}
\begin{tabular}{|c c c c c c c c c c c c|}
 \hline
$e_1$ & \textcolor{white}{$e_2$} & \textcolor{white}{$e_3$} & $e_4$ &$e_5$ & $e_6$ &$e_7$ & $e_8$ &$e_9$ & $e_{10}$ &$e_{11}$ & $e_{12}$   \\
 \hline
(2,1) & \textcolor{white}{(1,3)} & \textcolor{white}{(3,1)} & (1,4) & (4,3) & (3,2) & (2,1) & (1,4) & (4,3) & (3,1) & (1,4) & (4,3) \\
 \hline
\end{tabular}
\bigskip
\subcaption*{(b) After removing edges $e_1$, $e_4$, $e_5$, and $e_6$ corresponding to $s_3$ from the path obtained in Step 1}
\begin{tabular}{|c c c c c c c c c c c c|}
 \hline
\textcolor{white}{$e_1$} & \textcolor{white}{$e_2$} & \textcolor{white}{$e_3$} & \textcolor{white}{$e_4$} & \textcolor{white}{$e_5$} & \textcolor{white}{$e_6$} &$e_7$ & $e_8$ &$e_9$ & $e_{10}$ &$e_{11}$ & $e_{12}$   \\
 \hline
\textcolor{white}{(2,1)} & \textcolor{white}{(1,3)} & \textcolor{white}{(3,1)} & \textcolor{white}{(1,4)} & \textcolor{white}{(4,3)} & \textcolor{white}{(3,2)} & (2,1) & (1,4) & (4,3) & (3,1) & (1,4) & (4,3) \\
 \hline
\end{tabular}
\bigskip
\subcaption*{(c) After removing edges $e_8$, $e_9$, and $e_{10}$ corresponding to $s_2$ from the path obtained in Step 2}
\begin{tabular}{|c c c c c c c c c c c c|}
 \hline
\textcolor{white}{$e_1$} & \textcolor{white}{$e_2$} & \textcolor{white}{$e_3$} & \textcolor{white}{$e_4$} & \textcolor{white}{$e_5$} & \textcolor{white}{$e_6$} &$e_7$ & \textcolor{white}{$e_8$} & \textcolor{white}{$e_9$} & \textcolor{white}{$e_{10}$} &$e_{11}$ & $e_{12}$   \\
 \hline
\textcolor{white}{(2,1)} & \textcolor{white}{(1,3)} & \textcolor{white}{(3,1)} & \textcolor{white}{(1,4)} & \textcolor{white}{(4,3)} & \textcolor{white}{(3,2)} & (2,1) & \textcolor{white}{(1,4)} & \textcolor{white}{(4,3)} & \textcolor{white}{(3,1)} & (1,4) & (4,3) \\
 \hline
\end{tabular}
\label{table}
\end{table}

\begin{rems}\label{standrem}
If the length of a path is atleast $\vert\G\vert+1$ long, then the path is always decomposable. Further, the set of simple loops in the standard decomposition of $\sigma^{(n+1)}$ contains the set of simple loops in the standard decomposition of $\sigma^{(n)}$.\\
The standard decomposition algorithm respects the direction of the path in accordance with the signal $\sigma$. If the time dependence of the path $\sigma^{(n)}$ is ignored, there are several ways of decomposing it into simple loops and an indecomposable path. For example, another decomposition of $\sigma^{(13)}$ given in Example~\ref{standard} is given by three simple loops $e_{10}e_2=s_4$, $e_1e_4e_5e_6=s_3$, $e_8e_9e_3=s_2$, and an indecomposable path $e_7e_{11}e_{12}$.
\end{rems}

\subsection{Classes of switching signals} \label{class}
We will consider two classes of switching signals in $\mathcal{S}_{\G}$:
\begin{eqnarray*}
\mathcal{S}_{\G}(\tau) &=& \{\sigma\in \mathcal{S}_\G\ \vert \ t_{n+1}-t_n\geq \tau, \ n\geq 0\},\\
\mathcal\mathcal{S}_{\G}(\tau_1,\cdots,\tau_p) &=& \left\{\sigma\in \mathcal{S}_\G\ \vert \ \textit{the total time spent by the signal $\sigma$ on each simple loop $s_i$ } \right.\\
&& \left. \textit{in the standard decomposition of $\sigma^{(n)}$, for all $n\geq 1$, is atleast} \ \tau_i, \ i\in\mathbf{p}\right\}.
\end{eqnarray*}
\noindent For signals in $\mathcal{S}_{\G}(\tau)$, $\tau$ is the dwell time. In $\mathcal\mathcal{S}_{\G}(\tau_1,\cdots,\tau_p)$, $\tau_i$ will be known as the \emph{simple loop dwell time} on $s_i$.

\begin{exam}\label{examnew1}
Let $\G$ and $\sigma\in\mathcal{S}_\G$ be as in~\ref{standard} with the standard decomposition of $\sigma^{(13)}$. If the signal $\sigma$ belongs to the class $\mathcal{S}_\G(\tau_1,\cdots,\tau_4)$ in $\mathcal{S}_\G$, then $\delta_2+\delta_3\geq \tau_4$, $\delta_1+\delta_2+\delta_5+\delta_6\geq \tau_3$, and $\delta_8+\delta_9+\delta_{10}\geq \tau_2$.
\end{exam}
%\begin{lemma}\label{thmtau}
%The switched system~\ref{main} with switching signal in $\mathcal{S}_{\G}(\tau)$ is stable if $\tau>\mu_\G$.
%\end{lemma}
%\begin{proof}
%If $\tau>0$ is the dwell time, then  
%\[ 
%\ln a_n^\sigma = \sum_{j=1}^{n-1} \left( \ln \Vert P_{\sigma_{j+1}}^{-1}P_{\sigma_{j}}\Vert - \lambda_{\sigma_j}(t_j-t_{j-1})\right)\leq \sum_{j=1}^{n-1} \left( \ln \Vert P_{\sigma_{j+1}}^{-1}P_{\sigma_{j}}\Vert - \lambda\tau\right).
%\]
%Thus, if $\tau>\frac{\ln \Vert P_{s}^{-1}P_{r}\Vert}{\lambda}$, for all $(r,s)\in\mathcal{E}(\mathcal{G})$, then $\lim_{t\ra\infty} \Vert x(t)\Vert = 0$.
%\end{proof}

\begin{rem}\label{dwell}
In the literature, several lower bounds on the dwell time $\tau$ are obtained. In Theorem 1,~\cite{KS}, it is proved that for $\mu_\G = \max_{(r,s)\in\mathcal{E}(\mathcal{G})}\frac{\ln \Vert P_{s}^{-1}P_{r}\Vert}{\lambda_r}$, if $\tau>\mu_\G$, then the switched system~\ref{main} with switching signal in $\mathcal{S}_{\G}(\tau)$ is stable. In~\cite{O1}, a tighter lower bound for $\tau$ in terms of the maximum cycle ratio is obtained, given by $\rho^*$ and $\rho_2^*$ (for planar systems), where 
\[
\rho^*=\max_{i\in\mathbf{p}}\dfrac{\sum_{(r,s)\in \mathcal{E}(s_i)} \ln \Vert P_{s}^{-1}P_{r}\Vert}{\sum_{(r,s)\in \mathcal{E}(s_i)}\lambda_{r}},\ \rho_2^*=\max_{(i,j)\in \mathcal{E}(\G)}\dfrac{\ln \Vert P_j^{-1}P_i\Vert + \ln\Vert P_i^{-1}P_j\Vert}{\lambda_i+\lambda_j}.
\]
For planar systems, $\rho_2^*\leq \rho^*$.
\end{rem}

\noindent For a loop $C$ in the graph $\mathcal{G}$, let $-\lambda_{C} = \max_{(i,j)\in\mathcal{E}(C)}\{-\lambda_i\}>0$, and define 
\begin{eqnarray*}
\nu_C&=&\frac{\sum_{(r,s)\in \mathcal{E}(C)} \ln \Vert P_{s}^{-1}P_{r}\Vert}{\lambda_{C}}.
\end{eqnarray*}

\noindent For each \textit{simple loop} $s_i$ in the graph $\mathcal{G}$, let $\nu_i=\nu_{s_i}$. Note that if $C=n_1s_1+n_2s_2+\cdots+n_ps_p$, then $\nu_C\geq \sum_{i=1}^p n_i \nu_i$. Moreover, if $C=n_is_i$, for some $i$, then $\nu_C=n_i \nu_i$. Further, if $C$ is a self-loop, then $\nu_C=0$. It should also be noted that 
\[
\rho^* \leq \max_{i\in\textbf{p}} \dfrac{\sum_{(r,s)\in \mathcal{E}(s_i)} \ln \Vert P_{s}^{-1}P_{r}\Vert}{\ell(s_i)\lambda_{s_i}}=\max_{i\in\textbf{p}} \dfrac{\nu_{i}}{\ell(s_i)}.
\]
Moreover if $\G$ has only one loop, say $s$, and $\lambda_i=\lambda_j$, for all $(i,j)\in \mathcal{E}(s)$, then $\ell(s)\rho^*=\nu_s$.\\

\noindent We now prove our main theorem which gives lower bounds on the simple loop dwell times $\tau_i$ for stability of the switched system~\ref{main} with signal $\sigma\in \mathcal\mathcal{S}_{\G}(\tau_1,\cdots,\tau_p)$.

\begin{thm}\label{thmloop}
The switched system~\ref{main} with switching signal in $\mathcal\mathcal{S}_{\G}(\tau_1,\cdots,\tau_p)$ is stable if for each $i\in\mathbf{p}$, $\tau_i>\nu_i$.
\end{thm}

\begin{proof}
The standard decomposition of $\sigma^{(n)}=\sigma_1\cdots \sigma_n$ is a disjoint union of simple loops $s_i$, $n_i^\sigma$ times, for $i\in\mathbf{p}$, and an indecomposable path $p_n^\sigma$ of length
atmost $k-1$ (where $k=\vert \G\vert$). As done in~\cite{O1}, we distribute the terms in $\ln a_n^\sigma$ for each loop $s_i$ and path $p_n^\sigma$ to obtain 
\begin{equation} \label{main2}
\ln a_n^\sigma = \sum_{j=1}^{n-1} \left( \ln \Vert P_{\sigma_{j+1}}^{-1}P_{\sigma_{j}}\Vert - \lambda_{\sigma_j}(t_j-t_{j-1})\right) = b_n^\sigma + \sum_{i=1}^{p} n_i^\sigma \left( \sum_{(r,s)\in s_i} \ln \Vert P_{s}^{-1}P_{r}\Vert - \lambda_{s_i} \tau_{i}\right), 
\end{equation} 
where $b_n^\sigma$ are the terms corresponding to the path $p_n^\sigma$. \\ 
If $\tau_{i}>\frac{\sum_{(r,s)\in \mathcal{E}(s_i)} \ln \Vert P_{s}^{-1}P_{r}\Vert}{\lambda_{s_i}}$, for $i=1,\cdots,p$, then each term in the bracket in~\ref{main2} is negative. Moreover, as $n\ra\infty$, the number of simple loops $n_i^\sigma\ra\infty$, for some $i=1,\cdots,p$ (since the number of nodes in $\mathcal{G}$ are finite and $t_n\ra \infty$, see~\ref{zeno}), and $b_n^\sigma$ is a finite quantity for each $\sigma$ and $n$. Hence $\lim_{t\ra\infty} \Vert x(t)\Vert = 0$.
\end{proof}

\begin{rems}
(1) In the proof of~\ref{thmloop}, the standard decomposition of $\sigma^{(n)}$ does not play any role, except for the definition of $\mathcal\mathcal{S}_{\G}(\tau_1,\cdots,\tau_p)$. For example, take the standard decomposition and the other decomposition of $\sigma^{(n)}$ given in Remark~\ref{standrem}. The total time spent on the simple loop $s_4=e_2e_3=e_{10}e_2$ will depend on the chosen edge $e_2$ or $e_{10}$. This gives more flexibility to choosing a signal to ensure stability of the switched system, but finding all such decompositions is not easy. Infact the difficulty increases with $n$. \\
(2) If $\max_{(r,s)\in\mathcal{E}(\mathcal{G})}\ln \Vert P_{s}^{-1}P_{r}\Vert\leq 0$, the switched system is always stable by~\ref{dwell}. Moreover, if the maximum is non-positive, then $\sum_{(r,s)\in\mathcal{E}(\mathcal{G})}\ln \Vert P_{s}^{-1}P_{r}\Vert\leq 0$, and thus the switched system is stable by~\ref{thmloop}. Thus, we will assume that \mbox{$\max_{(r,s)\in\mathcal{E}(\mathcal{G})}\ln \Vert P_{s}^{-1}P_{r}\Vert> 0$} to obtain a non-trivial result. \\
%With this assumption, if $\sigma\in \mathcal{S}_{\G}(\tau)\cap \mathcal\mathcal{S}_{\G}(T_1,\cdots,T_p)$, then $T_i\geq \ell(s_i)\tau$, for $i\in\mathbf{p}$. Since 
%\[
%\nu_i\leq \ell(s_i)\frac{\max_{(r,s)\in\mathcal{E}(\mathcal{G})}\ln \Vert P_{s}^{-1}P_{r}\Vert}{\lambda}= \ell(s_i)\mu_\G,
%\]
%if $\tau>\mu_\G$, then $\nu_i<\ell(s_i)\tau\leq T_i$, for $i\in\mathbf{p}$. Hence the hypothesis of~\ref{thmloop} is satisfied, and therefore the switched system is stable. Thus, for signals in $\mathcal{S}_{\G}(\tau)\cap \mathcal\mathcal{S}_{\G}(T_1,\cdots,T_p)$, \ref{thmtau} \textbf{is a corollary} to~\ref{thmloop}.\\
(3) Since~\ref{thmloop} gives a lower bound on the total time spent on each simple loop, the signal can switch slowly on some edges on a simple loop, and faster on some of its other edges. This gives rise to signals with a combination of slow and fast switching. \\
(4) Along the lines of the proof of~\ref{thmloop}, we see that if the total time $\tau_c$ spent by the signal $\sigma$ on every loop $c$ in $\mathcal{G}$ satisfies $\tau_c > \nu_c$, then the switched system~\ref{main} is stable. Thus, the signal can adjust fast switches on some simple loops by switching slowly on its other constituting simple loops.  \\
(5) If each $A_i$ is a diagonal matrix, then $P_i=I$, hence $\nu_i=0$. Therefore, the switched system will always be stable for any switching signal. This can be directly seen from the first equality in~\ref{main2'}.
\end{rems}

\begin{rem}\label{unit} 
For $j=1,\cdots,k$, let $P_j'$ be another matrix of unit norm eigenvectors of $A_j$, that is $A_j=P_j'D_jP_j'^{-1}$, then there exists a unitary matrix $U_j$ such that $P_j'=P_jU_j$. Since the spectral norm is unitarily invariant, $\Vert P_i^{-1}P_j\Vert = \Vert U_i^{-1}P_i'^{-1}P_jU_j\Vert = \Vert P_i'^{-1}P_j'\Vert$, for $i,j=1,\cdots,k$. Thus the bounds obtained above will not depend on the choice of eigenvector matrix $P_j$ with unit norm eigenvectors. This remark is applicable throughout this paper. An appropriate scaling of eigenvector matrices will be used in Section~\ref{approach2} to obtain meaningful results.
This observation was used in~\cite{O1} and~\cite{KS} to obtain tighter bounds on the dwell time (in the case of all stable subsystems).
\end{rem}

%\begin{exam}
%Theorem~\ref{thmloop} is particularly useful when $\G$ has only one loop $s_1$, and the real part of the eigenvalues in every simple loop are comparable. Let $\G$ be the graph given in~\ref{adj_exam1new}. Consider a switched system~\ref{main} defined on $\mathbb{R}^2$, with a $\G$-admissible signal $\sigma$, with  
%\begin{eqnarray*}
%A_1 =\left(
%\begin{array}{c@{\enskip}c}
%-0.2&0.01\\
%0.01&-0.2
%\end{array}\right), 
%A_2 =\left(
%\begin{array}{c@{\enskip}c}
%-0.2&0.02\\
%0&-0.2
%\end{array}\right), 
%A_3 = \left(
%\begin{array}{c@{\enskip}c}
%-0.2&-0.01\\
%-0.01&-0.2
%\end{array}\right).
%\end{eqnarray*}
%
%\begin{figure}[h!]
%\centering
%\includegraphics[width=.3\textwidth]{graph_new.jpg}
%\caption{The underlying graph $\G$}
%\label{adj_exam1new}
%\end{figure}
%
%\noindent $\G$ has only one loop $s_1=\{v_1\ra v_3\ra v_2\ra v_1\}$ with $\nu_1 =20.5896$. The maximum cycle ratio as given in~\cite{O1} is given by $\rho^*=6.74487$ (which gives the minimum dwell time). If $\rho^*$ is taken as the minimum dwell time, then the total time spent on $s_1$ is atleast $3\rho^*=20.2346$. But, our method gives a great advantage, since on some edges of $s_1$, the signal can spend time less than $\rho^*$, while spending atleast $\nu_1$ time on the entire loop $s_1$. Note that the above hypothesis (H1) implies that if $\G$ has only one loop, then the switching sequence $(\sigma_n)_{n\geq 1}$ of $\sigma$ eventually follows this loop. \\
%\end{exam}

\begin{exam}
Let $\G$ be the graph given in~\ref{adj_exam1}. There are three simple loops in $\G$, namely $s_1=\{v_1\ra v_2\ra v_3\ra v_1\}$, $s_2=\{v_3\ra v_4\ra v_3\}$, and $s_1=\{v_3\ra v_5\ra v_6\ra v_3\}$. Consider a switched system~\ref{main} on $\mathbb{R}^2$, with a $\G$-admissible signal $\sigma\in\mathcal{S}_\G(\tau_1,\tau_2,\tau_3)$, with subsystem matrices   
\begin{eqnarray*}
& & A_1 =\left(
\begin{array}{c@{\enskip}c}
-1.5&0\\
0&-1.5
\end{array}\right), 
A_2 =\left(
\begin{array}{c@{\enskip}c}
-1&0\\
1&-1
\end{array}\right), \ A_3 = \left(
\begin{array}{c@{\enskip}c}
-11&3\\
-18&4
\end{array}\right), \\
& & A_4= \left(
\begin{array}{c@{\enskip}c}
3&-45\\
1&-11
\end{array}\right), 
A_5 = \left(
\begin{array}{c@{\enskip}c}
3&-46\\
1&-11
\end{array}\right),
A_6 = \left(
\begin{array}{c@{\enskip}c}
-2.1&1\\0&-2.1
\end{array}\right).
\end{eqnarray*}
%The underlying graph $\G$ (see~\ref{adj_exam1}) determining switching signal has adjacency matrix 
%\[
%B=\left(
%\begin{array}{c@{\enskip}c@{\enskip}c@{\enskip}c}
%1&0&1&0\\
%1&0&0&1\\
%0&1&0&0\\
%0&0&1&0
%\end{array}\right).
%\] 
\begin{figure}[h!]
\centering
\includegraphics[width=.5\textwidth]{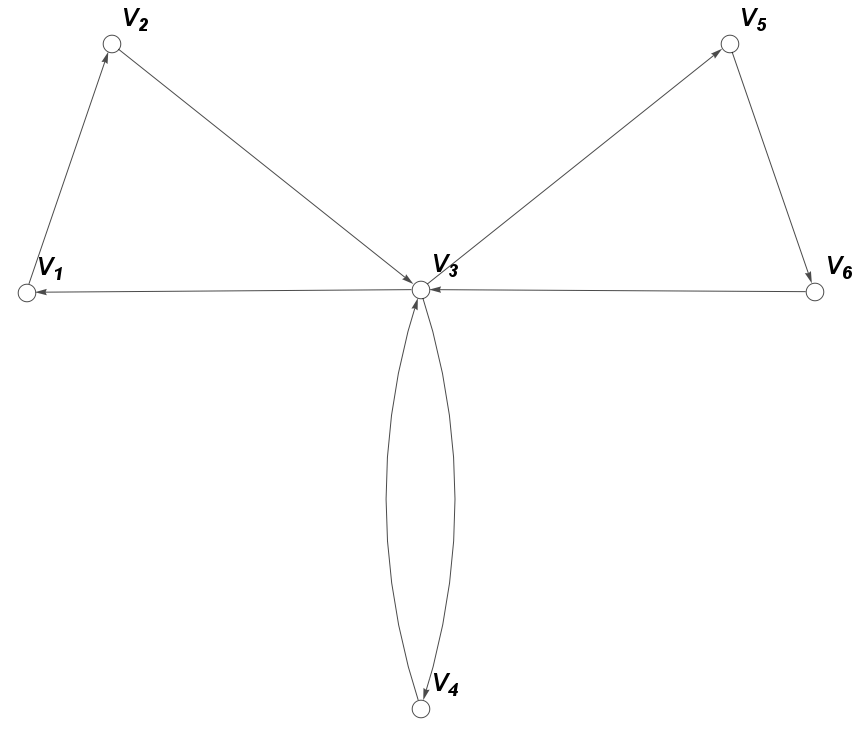}
\caption{The underlying graph $\G$}
\label{adj_exam1}
\end{figure}

\noindent Here $\nu_1 =2.64412$, $\nu_2 =2.73448$, and $\nu_3=2.89594$. By~\ref{thmloop}, the switched system is stable if $\tau_i>\nu_i$, for $i=1,2,3$. For this planar system, $\rho_2^*=1.32088$, $\rho^*=1.36724$. If $\rho_2^*$ is taken as the minimum dwell time, then the total time spent on simple loop $s_3$ is atleast $3\rho_2^*=3.96264$, which is greater than $\nu_3$.
\end{exam}

\section{Switched system with both stable and unstable subsystems}\label{stable-unstable}

In this section, we will consider graph-dependent switched systems which have both stable and unstable subsystems. Consider the switched system~\ref{main} with the following hypothesis, along with (H1):\\
\noindent \textbf{(H2)} $A_1,\cdots,A_r$ are stable diagonalizable matrices (over $\mathbb{C}$) and $A_{r+1},\cdots,A_k$ are unstable diagonalizable matrices (over $\mathbb{C}$). \\

\noindent See~\ref{diag} for diagonalizability hypothesis. For $i=1,\cdots,r$, $j=r+1,\cdots,k$, let $-\lambda_i=\max\{\textit{real part of eigenvalues of}\ A_i\}$, and $\mu_j=\max\{\textit{real part of eigenvalues of}\ A_j\}$. \\

\noindent The next example shows that a graph-dependent switched system can be stable even if some of its subsystems are unstable. 

\begin{exam}\label{unidir}
Let $\mathcal{G}$ be a unidirectional ring with two vertices. Let $A_1$ be a stable and $A_2$ be an unstable matrix, both diagonalizable (over $\mathbb{C}$). Let $-\lambda<0$ be the real part of the eigenvalue of $A_1$ closest to the imaginary axis. Let $\mu>0$ be the real part of the eigenvalue of $A_2$ farthest from the imaginary axis. For any $\G$-admissible switching signal $\sigma$, $(\sigma_n)_{n\geq 1}=1212\cdots$, or $2121\cdots$. In particular, let $(\sigma_n)_{n\geq 1} = 1212\cdots$, and assume that for every $m\geq 0$, $t_{2m+1}-t_{2m}\geq \tau$ and $t_{2m+2}-t_{2m+1}\leq \eta$, for some $\tau>0$ and $\eta>0$. \\
For $n$ even and $t\in [t_{n},t_{n+1})$, 
\begin{equation} \label{main3}
\Vert x(t)\Vert = \Vert e^{A_{1}(t-t_{n})}e^{A_{2}(t_{n}-t_{n-1})}\cdots e^{A_{1}t_{1}} x(0)\Vert \leq  \alpha e^{-\lambda\tau} \rho^{n/2} e^{n\beta/2}\Vert x(0)\Vert,
\end{equation}
and for $n$ odd and $t\in [t_{n},t_{n+1})$,
\begin{equation} \label{main4}
\Vert x(t)\Vert = \Vert e^{A_{2}(t-t_{n})}e^{A_{1}(t_{n}-t_{n-1})}\cdots e^{A_{1}t_{1}} x(0)\Vert\leq  \alpha e^{\beta} \rho^{n/2} e^{n\beta/2}\Vert x(0)\Vert,
\end{equation}
where, $\rho=\Vert P_2^{-1} P_1\Vert\Vert P_1^{-1} P_2\Vert\geq 1$, $\alpha = \max\{\Vert P_2\Vert \Vert P_1^{-1}\Vert, \Vert P_1\Vert \Vert P_1^{-1}\Vert\}$, and $\beta = -\lambda\tau+\mu\eta<0$. \\
The right hand side of~\ref{main3} and~\ref{main4} goes to 0, as $n\ra\infty$, if 
\begin{equation}\label{cond12new}
\ln\rho+\beta<0.
\end{equation}
This condition will be obtained for the signal $(\sigma_n)_{n\geq 1}=2121\cdots$ as well.\\

\noindent For the special case, when $A_1$ (stable) and $A_2$ (unstable) commute with each other, there exists an invertible matrix $P$ that simultaneously diagonalizes $A_1$ and $A_2$. Thus, $\rho=1$ (notation as above). For simplicity, assume that all the eigenvalues of both $A_1$ and $A_2$ are real (the complex case is more technical, but similar). Let $D_1=diag(\alpha_1,\cdots,\alpha_n)$, and $D_2=diag(\beta_1,\cdots,\beta_n)$ be such that $A_i=PD_iP^{-1}$, $i=1,2$. Then for the $\G$-admissible signal $\sigma\in\mathcal{S}_\G(\tau,\eta)$, the switched system~\ref{main} is stable if $\max\{\alpha_i\tau+\beta_i\eta\}<0$, for all $i=1,\cdots,n$. Further, if one of the $\beta_i$ is negative, then the corresponding condition is true for any choice of $\eta,\tau>0$, since $\alpha_i<0$. Note that the condition~\ref{cond12new} imply the set of conditions obtained here, since $\alpha_i\leq -\lambda$ and $\beta_i\leq \mu$. Moreover, if for the same index $i$, $\alpha_i= -\lambda$ and $\beta_i= \mu$, then the conditions obtained are same as~\ref{cond12new}. 
\end{exam}

\noindent For an arbitrary graph $\G$ with $k$ vertices, let $\sigma$ be a $\G$-admissible signal, and let the hypothesis (H2) be satisfied. Observe that if there exist a $T>0$ such that $\sigma(t)\in\{1,\cdots,r\}$, for all $t\geq T$, then the results from~\ref{stable} are applicable for the switched system for $t\geq T$. Moreover, if one of the vertices $v_i$, for $i\in \{r+1,\cdots,k\}$ does not have a (outgoing) directed path to any of the vertices $\{v_1,\cdots,v_r\}$, then for any $\G$-admissible signal which assumes the value $i$, the corresponding switched system will not be asymptotically stable. \\

\noindent In view of these observations, we will assume that the underlying graph $\G$ satisfies the following: for each $i\in\{r+1,\cdots,k\}$, there exist $j\in\{1,\cdots, r\}$ (depending on $i$) such that there is a path from $v_i$ to $v_j$, and for each $i\in\{1,\cdots,r\}$, there exist $j\in\{r+1,\cdots, k\}$ (depending on $i$) such that there is a path from $v_i$ to $v_j$. Moreover, we will assume that the $\G$-admissible signal $\sigma$ satisfies the following hypothesis:\\
\noindent \textbf{(H3)} For every $T>0$, there exists $t, s>T$ such that $\sigma(t)\in\{1,\cdots,r\}$ and $\sigma(s)\in\{r+1,\cdots,k\}$.

\subsection{Classes of Switching Signals}
Let us define the following collection of signals in $\mathcal{S}_{\G}$ satisfying (H2): for $\tau>0$ and $\eta>0$,
\begin{eqnarray*}
\mathcal{S}_{\G}(\tau,\eta) &=& \{\sigma\in \mathcal{S}_\G\ \vert \ t_{n+1}-t_n\geq \tau,\ \textit{if}\ \sigma(t_n)\in \{1,\cdots,r\},\ \textit{and}\\
& & \ t_{n+1}-t_n\leq \eta,\ \textit{if}\ \sigma(t_n)\in \{r+1,\cdots,k\}\}.
\end{eqnarray*}

\noindent For signals in $\mathcal{S}_{\G}(\tau,\eta)$, $\tau$ is known as the \textit{dwell time}, which is the minimum time that the signal spends in a stable subsystem, and $\eta$ will be called the \textit{flee time}, which is the maximum time that the signal spends in an unstable subsystem. \\

\noindent We now define another collection of signals in $\mathcal{S}_{\G}$ satisfying (H1) and (H2) in terms of the simple loops in $\G$. We know that each path has a standard decomposition into simple loops and an indecomposable path, see Section~\ref{stdec}. Let $s_1,\cdots,s_p$ be the simple loops in $\G$. Then each $s_i$ can have stable and unstable vertices (corresponding to stable and unstable subsystems, respectively). For a signal $\sigma$, let $\sigma^{(n)}=\sigma_1\cdots\sigma_n$ (as before). For every $n\geq 1$, under the standard decomposition of $\sigma^{(n)}$, let the signal spends atleast $\tau_i$ time on the stable vertices of $s_i$, and atmost $\eta_i$ time on the unstable vertices of $s_i$, $i\in\textbf{p}$. Let $\mathcal{S}_{\mathcal{G}}(\tau_1,\eta_1,\cdots,\tau_p,\eta_p)$ denote the collection of all such signals.

\begin{exam}\label{exam42}
Let $\G$ and $\sigma\in\mathcal{S}_\G$ be as in~\ref{standard} with the standard decomposition of $\sigma^{(13)}$. Let $A_1, A_2$ be stable matrices and $A_3, A_4$ be unstable matrices satisfying the usual diagonalizability hypothesis (H2). If the signal $\sigma$ belongs to the class $\mathcal{S}_\G(\tau_1,\eta_1,\cdots,\tau_4,\eta_4)$ in $\mathcal{S}_\G$, then $\delta_2\geq \tau_4$, $\delta_3\leq \eta_4$, $\delta_1+\delta_2\geq \tau_3$, $\delta_5+\delta_6\leq \eta_3$, and $\delta_8+\delta_9+\delta_{10}\leq \eta_2$.
\end{exam}

\subsection{Unidirectional Ring} \label{uni}
\ref{unidir} can be generalized to a unidirectional ring $\mathcal{G}$ with $k$ vertices and $A_1,\cdots,A_k$ satisfying hypothesis (H2). 

\begin{prop}\label{unidirthmnew1}
With the notation as above, the switched system~\ref{main} with $\G$-admissible signal $\sigma\in\mathcal{S}_{\G}(\tau,\eta)$ is stable, if $\eta>0$ and $\tau>0$ satisfy
\[
\ln\rho+\beta<0,
\]
where $\rho=\prod_{(i,j)\in\mathcal{E}(\mathcal{G})}\Vert P_j^{-1} P_i\Vert$, $\beta=-\left( \sum_{i=1}^r\lambda_i\right)\tau + \left(\sum_{j=r+1}^k \mu_j\right) \eta$. 
\end{prop}

\noindent The unidirectional ring $\G$ has only one simple loop, which is $\G$ itself. Let $s_1=\G$.

\begin{prop}\label{unidirthmnew2}
With the notation as above, the switched system~\ref{main} with $\G$-admissible signal $\sigma\in \mathcal{S}_{\G}(\tau_1,\eta_1)$ is stable, if $\tau_1,\eta_1>0$ satisfy
\[
\ln\rho-\lambda \tau_1+\mu \eta_1<0,
\]
where $\rho=\prod_{(i,j)\in\mathcal{E}(\mathcal{G})}\Vert P_j^{-1} P_i\Vert$, $-\lambda=\max\{-\lambda_1,\cdots,-\lambda_r\}$, and $\mu=\max\{\mu_{r+1},\cdots,\mu_k\}$. 
\end{prop}

\begin{exam}
Let $\G$ be the unidirectional ring $\{v_1\ra v_2\ra v_3\ra v_1\}$. Consider a switched system~\ref{main} on $\mathbb{R}^2$, with a $\G$-admissible signal $\sigma\in \mathcal{S}_{\G}$, with
\begin{eqnarray*}
A_1 =\left(
\begin{array}{c@{\enskip}c}
-2&0\\
0&-2
\end{array}\right), 
A_2 =\left(
\begin{array}{c@{\enskip}c}
-0.4&-0.03\\
1.43&0.4
\end{array}\right), 
A_3 = \left(
\begin{array}{c@{\enskip}c}
0.9&0\\
2&-0.6
\end{array}\right).
\end{eqnarray*}
%Let the underlying graph $\G$ be the unidirectional ring as shown in~\ref{adj_exam2}. 
%
%\begin{figure}[h!]
%\centering
%\includegraphics[width=.5\textwidth]{adj_exam2.jpg}
%\caption{The underlying graph $\G$ - unidirectional ring.}
%\label{adj_exam2}
%\end{figure}
\noindent Here $A_1, A_2$ are stable and $A_3$ is unstable. According to the~\ref{unidirthmnew1}, the switched system~\ref{main} with the switching signal $\sigma\in \mathcal{S}_{\G}(\tau,\eta)$ is stable if $\eta>0, \tau>0$ satisfy $3.38306-2.1 \tau + 0.9 \eta<0$.
\noindent Further, using~\ref{unidirthmnew2}, the switched system~\ref{main} with the switching signal $\sigma\in \mathcal{S}_{\G}(\tau_1,\eta_1)$ is stable if $\tau_1,\eta_1>0$ satisfy $3.38306-0.1 \tau + 0.9 \eta<0$.
\end{exam}

%\begin{rem}
%Let $\sigma$ is a \textit{periodic signal} with period $\ell$. That is $\sigma_{n+\ell}=\sigma_n$, for all $n\geq 1$. Let $\ell_s$ be the number of times the switched system assumes a stable subsystem in $\ell$ consecutive switches. \\
%For $t>0$, $t=x\ell+y$, where $x,y$ are non-negative integers with $0\leq y<\ell$. Thus, $N_s(0,t)=x\ell_s+\min\{\ell_s,y\}$, and $N_u(0,t)=x\ell_u+\max\{0,y-\ell_s\}$. Hence
%\[
%\dfrac{x\ell_u}{x\ell_s+\ell_s}\leq \dfrac{N_u(0,t)}{N_s(0,t)}\leq \dfrac{x\ell_u+\ell}{x\ell_s}.
%\]
%Thus, $\lim_{t\rightarrow\infty}\dfrac{N_u(0,t)}{N_s(0,t)} = \dfrac{\ell_u}{\ell_s}$. \\
%If $\ln\rho_2\geq 0$, then for $\tau>\ln\rho_1/\lambda$, the inequality~\ref{su} is satisfied if
%\[
%\dfrac{\ell_u}{\ell_s}<\gamma,
%\]
%where $\gamma=\dfrac{-(\ln\rho_1-\lambda\tau)}{\ln\rho_2+\mu\eta}$. Further, this is possible only if $\ell\geq 1+\dfrac{1}{\gamma}$.\\
%
%\noindent For $\ln\rho_2<0$, if $\tau>0$, $\eta>0$, and $\ell_s, \ell_u\geq 1$ satisfy
%\begin{eqnarray}\label{per}
%\ell_u(\ln\rho_2+\mu\eta)+\ell_s(\ln\rho_1-\lambda\tau) &<& 0,
%\end{eqnarray}
%then the switched system~\ref{main} with periodic switching signal $\sigma\in\mathcal{S}_{\G}(\tau,\eta)$ is stable. Note that for a unidirectional ring $\G$, every $\G$-admissible switching signal is periodic. Thus, the results given in Sections~\ref{uni} and~\ref{bip} follow from the preceding arguments about periodic signals.
%\end{rem}

\subsection{Bipartite Graph}\label{bip}
In this section, we state stability results of the switched system~\ref{main} with a bipartite underlying graph $\G$, with stable and unstable vertices (corresponding to stable and unstable subsystems).

\begin{prop}\label{bipthm}
Let $\mathcal{G}$ be a bipartite graph with disjoint classes $\{v_1,\cdots,v_r\}$ and \\
$\{v_{r+1},\cdots,v_k\}$. Consider the switched system~\ref{main} with $\sigma\in\mathcal{S}_{\G}(\tau,\eta)$, and the subsystems satisfying hypothesis (H2). The switched system is stable if
\[
\ln\rho_1+\ln\rho_2+\beta<0,
\]
where $\rho_!=\max_{(i,j)\in\mathcal{E}(\G), i=1,\cdots,r}\Vert P_j P_i^{-1}\Vert$, $\rho_2=\max_{(i,j)\in\mathcal{E}(\G), i=r+1,\cdots,k}\Vert P_j P_i^{-1}\Vert$, $\beta = -\lambda\tau+\mu\eta$, $-\lambda=\max\{-\lambda_1,\cdots,-\lambda_r\}$, and $\mu=\max\{\mu_{r+1},\cdots,\mu_k\}$.
\end{prop}

\begin{rem}
The proof of Proposition~\ref{bipthm} is similar to the case of unidirectional ring with two vertices considered in~\ref{unidir}. This is because all the stable subsystems and unstable subsystems are in the two disjoint classes of the underlying bipartite graph $\G$.
\end{rem}

%\begin{exam}
%\textcolor{red}{Let  the switched system}~\ref{main} be defined on $\mathbb{R}^2$ and comprises of four subsystems \begin{eqnarray*}
%& & A_1 =\left(
%\begin{array}{c@{\enskip}c}
%-0.2&0\\
%0&-0.2
%\end{array}\right), 
%A_2 =\left(
%\begin{array}{c@{\enskip}c}
%-0.44&-0.03\\
%1.43&0.04
%\end{array}\right), \\
%& & A_3 = \left(
%\begin{array}{c@{\enskip}c}
%19&-7\\
%42&-16
%\end{array}\right), 
%A_4 = \left(
%\begin{array}{c@{\enskip}c}
%0.1667&9.1667\\
%-0.1667&2.8333
%\end{array}\right).
%\end{eqnarray*}
%Let the underlying graph $\G$ (see~\ref{adj_exam3}) be bipartite ($r=2$). 
%
%\begin{figure}[h!]
%\centering
%\includegraphics[width=.5\textwidth]{adj_exam3.jpg}
%\caption{The underlying graph $\G$ - bipartite}
%\label{adj_exam3}
%\end{figure}
%
%\noindent Here $\ln\rho=2.927$. The switched system in $\mathcal{S}_{\G}(\tau,\eta)$ is stable if $\eta>0, \tau>0$ satisfy
%\[
%5.855- 0.08 \tau + 5\eta<0.
%\]
%\end{exam}

\subsection{An arbitrary graph}\label{gen} 
Let $\G$ to be a digraph with $k$ vertices $\{v_1,\cdots,v_k\}$ with has $p$ simple loops $s_1,\cdots,s_p$. Let $\G_s$ and $\G_u$ be subgraphs of $\G$ with
\[
\E(\G_s) = \{(i,j)\in \E(\G)\ \vert\ i = 1,\cdots ,r \},\ \E(\G_u) = \{(i,j)\in \E(\G)\ \vert\ i=r+1,\cdots,k\}.
\]
Then $\G$ is a superimposition of $\G_s$ and $\G_u$. For $k=5$, $r=2$, an example of $\G$ with corresponding $\G_s$ and $\G_u$ is shown in Figure~\ref{adj_exam4}. Let the matrices $A_1,\cdots,A_k$ satisfy the hypothesis (H2). We obtain the following stability results.

\begin{prop}\label{bipthmnew1}
Consider the switched system~\ref{main} with $\sigma\in\mathcal{S}_{\G}(\tau,\eta)$, and the subsystems satisfying hypothesis (H2). The switched system is stable if for all $i=1,\cdots,p$,
\begin{equation}\label{bip1}
\ln\rho_i- \lambda_{s_i} \tau+ \mu_{s_i}\eta < 0,
\end{equation}
where, 
\[
\rho_i=\prod_{(u,v)\in \mathcal{E}(s_i)}\Vert P_v^{-1} P_u\Vert,\ \lambda_{s_i}=\sum_{(u,v)\in \mathcal{E}(s_i)\cap \mathcal{E}(\G_s)}\lambda_u, \ \text{and} \ \mu_{s_i}=\sum_{(u,v)\in \mathcal{E}(s_i)\cap \mathcal{E}(\G_u)}\mu_u.
\]
\end{prop}

\begin{prop}\label{bipthmnew2}
Consider the switched system~\ref{main} with $\sigma\in\mathcal{S}_{\G}(\tau_1,\eta_1,\cdots,\tau_p,\eta_p)$, and subsystems satisfying hypothesis (H2). The switched system is stable if for all $i=1,\cdots,p$,
\begin{equation}\label{bip2}
\ln\rho_i-\lambda_{s_i}\tau_i+\mu_{s_i}\eta_i  < 0,
\end{equation}
where, 
\[
\rho_i=\prod_{(u,v)\in \mathcal{E}(s_i)}\Vert P_v^{-1} P_u\Vert,\ -\lambda_{s_i}=\max_{(u,v)\in \mathcal{E}(s_i)\cap \mathcal{E}(\G_s)}-\lambda_u, \ \text{and} \ \mu_{s_i}=\max_{(u,v)\in \mathcal{E}(s_i)\cap \mathcal{E}(\G_u)}\mu_u.
\]
\end{prop}

\begin{figure}[htp]
\centering
\includegraphics[width=.32\textwidth]{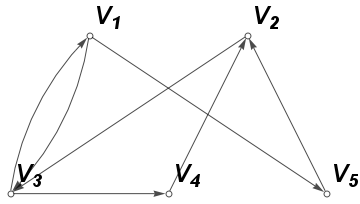}\hfill
\includegraphics[width=.32\textwidth]{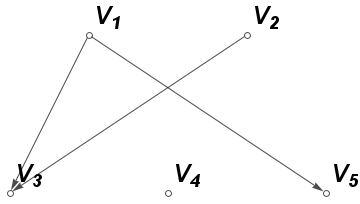}\hfill
\includegraphics[width=.32\textwidth]{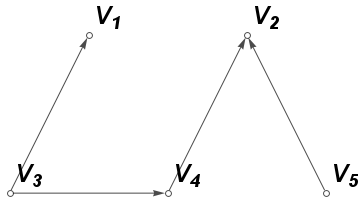}
\caption{(Left) Graph $\G$, (Center) Subgraph $\G_s$ of $\G$, (Right) Subgraph $\G_u$ of $\G$}
\label{adj_exam4}

\end{figure}

\begin{rem}
The proofs of Propositions~\ref{bipthmnew1} and \ref{bipthmnew2} follow from standard inequalities used before. Conditions~\ref{bip1} and~\ref{bip2} coincide for Example~\ref{unidir}, and we obtain~\ref{cond12new}. In~\ref{exam42}, we obtained the inequalities $\delta_1+\delta_2\geq \tau_3$ and $\delta_5+\delta_{6}\leq \eta_3$. Thus, if for the system given in that example, all the inequalities in~\ref{bip2} are satisfied, then the switched system is stable. Hence, the signal can switch non-uniformly on $e_1$ and $e_2$, while satisfying $\delta_1+\delta_2\geq \tau_3$, and also switch non-uniformly on $e_5$ and $e_6$, while satisfying $\delta_5+\delta_{6}\leq \eta_3$. This gives an advantage over setting a uniform dwell time and flee time, as obtained in~\ref{bipthmnew1}. 
\end{rem}

We now study stability of the switched system~\ref{main} with a $\G$-admissible signal $\sigma\in \mathcal{S}_{\G}(\tau,\eta)$ with two different approaches. The first approach, given in Section~\ref{approach1}, is to find sufficient conditions on the switching pattern of the signal $\sigma$ to ensure stability. In the second approach, given in Section~\ref{approach2}, it is assumed that $\G_u$ is acyclic, and we appropriately scale the eigenvector matrices $P_1,\cdots,P_k$ of $A_1,\cdots,A_k$, respectively (see~\ref{unit}), and find sufficient conditions on the dwell time $\tau>0$, and the flee time $\eta>0$ so that for each signal $\sigma\in\mathcal{S}_\G(\tau,\eta)$, the switched system~\ref{main} is stable. This uses the concept of topological sorting for acyclic graphs.

\subsubsection{The first approach}\label{approach1}
Using the first inequality of~\ref{main2'}, for $t\in [t_{n-1},t_n)$,
\begin{equation} \label{avdwell}
\begin{aligned}
\Vert x(t)\Vert &= \Vert e^{A_{\sigma_n}(t-t_{n-1})}e^{A_{\sigma_{n-1}}(t_{n-1}-t_{n-2})}\cdots e^{A_{\sigma_{1}}t_{1}} x(0)\Vert  \\ 
&\leq  \Vert P_{\sigma_n}\Vert \Vert P_{\sigma_1}^{-1}\Vert e^{-\lambda_{\sigma_n}(t-t_{n-1})}\left(\prod_{j=1}^{n-1}\Vert
P_{\sigma_{j+1}}^{-1}P_{\sigma_{j}}\Vert e^{-\lambda_{\sigma_j}(t_j-t_{j-1})}\right)\Vert x(0)\Vert \\ 
&\leq  C \left(e^{-\lambda_{\sigma_n}(t-t_{n-1})}\prod_{j=1}^{n}\Vert
P_{\sigma_{j+1}}^{-1}P_{\sigma_{j}}\Vert e^{-\lambda_{\sigma_j}(t_j-t_{j-1})}\right)\Vert x(0)\Vert \\ 
&\leq  C \rho_1^{N_s(0,t)}\rho_2^{N_u(0,t)} e^{-\lambda \tau N_{s}(0, t) + \mu \eta N_{u}(0, t)}\Vert x(0)\Vert,
\end{aligned}
\end{equation} 
where $C$ is a positive constant, $\rho_1=\max\{\Vert P_j^{-1} P_i\Vert\ \vert\ (i,j)\in\mathcal{E}(\G_s)\}$,\\ $\rho_2=\max\{\Vert P_j^{-1} P_i\Vert\ \vert\ (i,j)\in\mathcal{E}(\G_u)\}$, $N_s(0,t)=\#\{i\ \vert\ 1\leq i\leq n, \sigma_i\in \{1,\cdots,r\}\}$, and $N_u(0,t)=\#\{i\ \vert\ 1\leq i\leq n, \sigma_i\in \{r+1,\cdots,k\}\}$. Thus, we obtain the following result.

\begin{prop}\label{thmsu}
With the notation as above, the switched system~\ref{main} with the signal $\sigma\in \mathcal{S}_{\G}(\tau,\eta)$, and subsystems satisfying (H2) is stable if 
\begin{equation} \label{su}
\limsup_{t\ra\infty} \left(N_s(0,t)(\ln \rho_1-\lambda \tau) + N_u(0,t)(\ln \rho_2 + \mu \eta )\right) < 0,
\end{equation} 
\end{prop}

\begin{rems}\label{app}
1) It is worth comparing Propositions~\ref{unidirthmnew1}, ~\ref{bipthm}, and~\ref{thmsu}.\\
2) Note that if the signal is such that $\sigma_n\in\{1,\cdots,r\}$, for all $n\geq N$, for some $N\geq 1$, then $N_u(0,t)\leq N$ and $\tau>\ln\rho_1/\lambda$ is necessary for~\ref{su}. Moreover, if the signal is such that $\sigma_n\in\{r+1,\cdots,k\}$, for all $n\geq N$, for some $N\geq 1$, then $N_s(0,t)\leq N$ and there is no choice of $\eta$ for which~\ref{su} is satisfied. See hypothesis (H3). \\
3) If $\rho_2\geq 1$, then for~\ref{su} to be satisfied, $\ln\rho_1-\lambda\tau$ must be negative, therefore $\tau>\ln\rho_1/\lambda$ is a necessary condition for~\ref{su}.  \\
4) The term in the big bracket of the second inequality in~\ref{avdwell} is bounded above by 
\[
\left(\prod_{\substack{j=1\\ \sigma_j\in\{1,\cdots,r\}}}^{n}\Vert
P_{\sigma_{j+1}}^{-1}P_{\sigma_{j}}\Vert e^{-\lambda_{\sigma_j}\tau}\right) \left(\prod_{\substack{j=1\\ \sigma_j\in\{r+1,\cdots,k\}}}^{n}\Vert
P_{\sigma_{j+1}}^{-1}P_{\sigma_{j}}\Vert e^{\mu_{\sigma_j}\eta}\right).
\] 
Thus, if $\rho_2<1$, then for 
\[
\tau>\max_{\substack{(i,j)\in\mathcal{E}(\G)\\
i\in\{1,\cdots,r\}}} \dfrac{\ln \Vert P_j^{-1}P_i\Vert}{\lambda_i}, \ \text{and}\ \eta<-\max_{\substack{(i,j)\in\mathcal{E}(\G)\\
i\in\{r+1,\cdots,k\}}} \dfrac{\ln \Vert P_j^{-1}P_i\Vert}{\mu_i},
\]
the switched system~\ref{main} with $\sigma\in\mathcal{S}(\tau,\eta)$ is stable. Note that these conditions imply~\ref{su}. Here we have obtained the lower bound $\mu_\G$ for $\tau$, as mentioned in~\ref{dwell}, when all the subsystems are stable. \\
5) If $\G$ is acyclic, then the hypothesis (H3) is not satisfied for any signal, since $\G$ will be a tree. Hence, for (H3) to be satisfied, $\G$ must necessarily contain a loop. See hypothesis (H1).\\
6) A similar result is obtained in~\cite{ZWXS} for switched positive linear systems with both stable and unstable subsystems.
%In this case,
%\[
%\limsup_{t\ra\infty}\frac{N_s(0,t)}{N_u(0,t)} > \frac{\eta\mu+\ln\rho}{\lambda\tau-\ln\rho} \geq 0
%\]
%implies~\ref{su}.
\end{rems}

\begin{exam}\label{gen_exam1}
Let  the switched system~\ref{main} be defined on $\mathbb{R}^2$ and comprises of five subsystems  %%%%%%%%%%%%%%%%%%
\begin{eqnarray*}
A_1 =\left(
\begin{array}{c@{\enskip}c}
-2&0\\
0&-2
\end{array}\right), \
A_2 =\left(
\begin{array}{c@{\enskip}c}
-0.4&-0.03\\
1.4&0.04
\end{array}\right), \
A_3 = \left(
\begin{array}{c@{\enskip}c}
1&1\\
0&1
\end{array}\right),
\end{eqnarray*}
\begin{eqnarray*}
A_4 = \left(
\begin{array}{c@{\enskip}c}
0.1&0\\
0.1&0.2
\end{array}\right),\
A_5 = \left(
\begin{array}{c@{\enskip}c}
2&0.1\\0.1&2
\end{array}\right).
\end{eqnarray*}

Here $A_1, A_2$ are stable, and $A_3, A_4, A_5$ are unstable subsystems. Let the underlying graph $\G$ be as in~\ref{adj_exam4}. Here $\ln\rho_1/\lambda=3.45021$ and $\ln\rho_2/\mu=1.21024$. Thus, taking $\tau=4$ and $\eta=1$, the inequality~\ref{su} becomes 
\[
\limsup_{t\ra\infty}\dfrac{N_s(0,t)}{N_u(0,t)}>\dfrac{\ln\rho_2+\mu\eta}{-(\ln\rho_1-\lambda\tau)} = 84.4229.
\]
\end{exam} 

\subsubsection{The second approach}\label{approach2}
%The second approach involves viewing the graph $\G$ as a superimposition of graphs. The \textit{superimposition of graphs} $\G_1,\cdots,\G_\ell$ with the same vertex set is a graph with vertex set same as $\G_1$ and whose edge set is the union of the edge sets of $\G_1,\cdots,\G_\ell$. Thus any graph $\G$ with vertex set $\{v_1,\cdots,v_k\}$ can be seen as a superimposition of four graphs (see~\ref{adj_exam4} for example), namely $\G_s$, $\G_u$ and bipartite graphs $\G_b^1$ and $\G_b^2$ with disjoint sets $\{1,\cdots,r\}$ and $\{r+1,\cdots,k\}$, where
%\begin{eqnarray*}
%\E(\G_s) &=& \{(i,j)\in \E(\G)\ \vert\ i,j = 1,\cdots ,r \},\\
%\E(\G_u) &=& \{(i,j)\in \E(\G)\ \vert\ i,j=r+1,\cdots,k\},\\
%\E(\G_b^1) &=& \{(i,j)\in \E(\G)\ \vert\ i=1,\cdots,r; j=r+1,\cdots,k\}, \\
%\E(\G_b^2) &=& \{(j,i)\in \E(\G)\ \vert\ i=1,\cdots,r; j=r+1,\cdots,k\}.
%\end{eqnarray*}

In this section, we assume that the subgraph $\G_u$ of $\G$ is acyclic and obtain bounds on $\tau$ and $\eta$ to ensure stability of the switched system~\ref{main} with all switching signals $\sigma\in\mathcal{S}_\G(\tau,\eta)$. Let $P_1,\cdots,P_k$ be any choice of matrices whose columns are eigenvectors of $A_1,\cdots,A_k$, respectively (not necessarily with unit norm). For $\sigma\in\mathcal{S}_{\G}(\tau,\eta)$, $t\in [t_n,t_{n+1})$,
\begin{eqnarray*} 
\Vert x(t)\Vert &=& \Vert e^{A_{\sigma_n}(t-t_{n})}e^{A_{\sigma_{n}}(t_{n}-t_{n-1})}\cdots e^{A_{\sigma_{1}}t_{1}} x(0)\Vert \nonumber \\ 
%&=& \Vert P_{\sigma_n}e^{D_{\sigma_n}(t-t_{n-1})}P_{\sigma_n}^{-1}\left(\prod_{j=1}^{n-1}P_{\sigma_{n-j}}e^{D_{\sigma_{n-j}}(t_{n-j}-t_{n-j-1})}P_{\sigma_{n-j}}^{-1}\right) x(0)\Vert \nonumber \\
&\leq & \Vert P_{\sigma_n}\Vert \Vert P_{\sigma_1}^{-1}\Vert \Vert e^{D_{\sigma_n}(t-t_{n})}\Vert \left(\prod_{j=1}^{n}\Vert
P_{\sigma_{j+1}}^{-1}P_{\sigma_{j}}\Vert \Vert e^{D_{\sigma_j} (t_j-t_{j-1})}\Vert\right)\Vert x(0)\Vert  \\ 
&\leq& C_1 C_2\ a_n^{\sigma} \Vert x(0)\Vert,
\end{eqnarray*} 
where $C_1 = \max\{\Vert P_{j}\Vert \Vert P_{i}^{-1}\Vert\ \vert\ \text{there is a path in $\G$ from $v_i$ to $v_j$}, (i,j)\in\mathbf{k}\}$, \\
$C_2=\max\{e^{-\tau \lambda},e^{\eta \mu}\}$, and $a_n^\sigma = \prod_{j=1}^{n}  \Vert P_{\sigma_{j+1}}^{-1}P_{\sigma_{j}}\Vert \Vert e^{D_{\sigma_j} (t_j-t_{j-1})}\Vert$. \\

\noindent The path $\sigma_1, \sigma_2, \cdots, \sigma_{n+1}$ can be decomposed into: edges in $\G_s$ and edges in $\G_u$. Distributing the terms in $\ln a_n^\sigma$, we get 
\begin{equation} \label{main5}
\ln a_n^\sigma \leq  N_n \sum_{(a,b)\in \G_s}\left( \ln \Vert P_{b}^{-1}P_{a}\Vert - \lambda_a \tau \right) + (n-N_n) \sum_{(a,b)\in \G_u }\left( \ln \Vert P_{b}^{-1}P_{a}\Vert + \mu_a \eta \right),
\end{equation}  %%%%%% 
where $N_n=\#\{(\sigma_j,\sigma_{j+1})\in\mathcal{E}(\G_s)\ \vert\ j\in\mathbf{n}\}$.\\

\noindent As $t\ra +\infty$, $n\ra \infty$, hence atleast one of $N_n$ or $n-N_n$ diverges. Thus, if each term in the summation on the right hand side of the inequality~\ref{main5} is negative, then as $t\ra +\infty$, $\Vert x(t)\Vert\ra 0$. \\

\noindent It is easy to see that if a graph is acyclic, then there is a vertex with zero indegree and a vertex with zero outdegree. A topological sorting of a digraph is linearly ordering the vertices such that if there is a directed edge from a vertex $v$ to vertex $w$, then $v$ comes after $w$ in the ordering. For a digraph, topological sorting is possible if the graph is acyclic. Lemma~\ref{acyclic} uses a concept of \textit{topological sorting}, we refer to~\cite{Cormen} for details. 

\begin{exam}
Consider the acyclic graph $\G_1$ given in Figure~\ref{top_sor}. Both $\{3,4,1,5,2,6\}$ and $\{3,4,6,5,1,2\}$ are topological sortings of $\G$.
\end{exam}

\begin{figure}[h!]
\centering
\includegraphics[width=.5\textwidth]{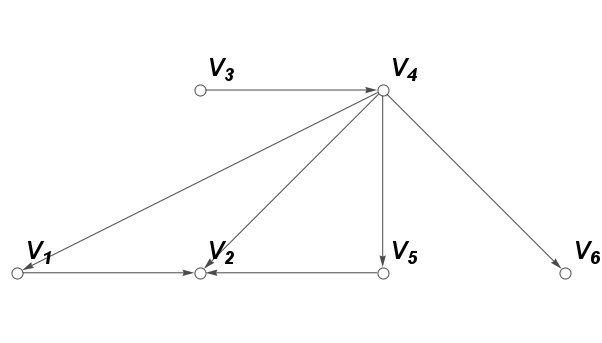}
\caption{The graph $\G_1$}
\label{top_sor}
\end{figure}

\begin{lemma}\label{acyclic}
If the subgraph $\G_u$ of $\G$ is acyclic, there are eigenvector matrices $Q_1,\cdots,Q_k$ of $A_1,\cdots,A_k$ so that 
\[
\max\{\Vert Q_b^{-1}Q_a\Vert \ \vert\ (a,b)\in \mathcal{E}(\G_u)\}<1.
\]
\end{lemma}
\begin{proof}
Suppose $\rho=\max\{\Vert P_b^{-1}P_a\Vert \ \vert\ (a,b)\in \mathcal{E}(\G_u)\}$ be greater than 1. We will choose an appropriate scaling of $P_1,\cdots,P_k$ to obtain eigenvector matrices $Q_1,\cdots,Q_k$ such that $\rho'=\max\{\Vert Q_b^{-1}Q_a\Vert \ \vert\ (a,b)\in \mathcal{E}(\G_u)\}$ is less than 1. Let $0<\zeta<1$ and $\tau=\rho/\zeta>1$. Let $\epsilon>0$ be fixed.\\
Since $\G_u$ is acyclic, it has a topological sorting, say $v_{a_1}\ra v_{a_2}\ra\cdots\ra v_{a_m}$. Let the linear sequence in the sorting be such that $v_{a_1},\cdots, v_{a_{m_1}}\in\{r+1,\cdots,k\}$ and have zero indegree, $v_{a_{m_1+1}},\cdots, v_{a_{m_2}}\in\{r+1,\cdots,k\}$ and $v_{a_{m_2+1}},\cdots, v_{a_{m}}\in\{1,\cdots,r\}$. Let $Q_{a_j}=P_{a_j}$, for $j=1,\cdots,m_1$, $Q_{a_j}=\tau^{j-m_1}P_{a_j}$, for $j=m_1+1,\cdots,m_2$, and $Q_{a_j}=\tau^{m_2-m_1+1}P_{a_j}$, for $j=m_2+1,\cdots,m$. For the remaining indices $i\in\textbf{k}\setminus\{a_1,\cdots,a_m\}$, set $Q_i=P_i$.\\
Note that $\Vert Q_b^{-1}Q_a\Vert = \tau^{-i} \Vert P_b^{-1}P_a\Vert$, for some $i\geq 1$. Hence  $\Vert Q_b^{-1}Q_a\Vert \leq \tau^{-1} \Vert P_b^{-1}P_a\Vert\leq \zeta<1$.
\end{proof}

\noindent If the graph $\G_u$ is acyclic, then using~\ref{acyclic}, choose the eigenvector matrices $Q_1,\cdots,Q_k$ of $A_1,\cdots,A_k$, and let  
\[
\rho'=\max\{\Vert Q_b^{-1}Q_a\Vert \ \vert\ (a,b)\in \mathcal{E}(\G_u)\}, \ \text{and } \alpha'=\max\{\Vert Q_b^{-1}Q_a\Vert \ \vert\ (a,b)\in \mathcal{E}(\G_s)\}.
\] 
Note that $\rho'<1$. Thus we obtain the following result.

\begin{prop}\label{thmsufinal}
With the notation as above, if the sub-graph $\G_u$ of $\G$ is acyclic, then the switched system~\ref{main} with $\sigma\in S_{\G}(\tau,\eta)$ is stable if 
\begin{equation} \label{cond1final}
\tau>\max_{(i,j)\in\mathcal{E}(\G_u)} \dfrac{\ln \Vert Q_j^{-1}Q_i\Vert}{\lambda_i}, \ \text{and}\ \eta<-\max_{(i,j)\in\mathcal{E}(\G_s)} \dfrac{\ln \Vert Q_j^{-1}Q_i\Vert}{\mu_i}.
\end{equation} 
\end{prop}

\begin{rem}
The proof of~\ref{thmsufinal} follows by replacing $P_i$ by $Q_i$ in~\ref{main5}. \\
A condition similar to~\ref{cond1final} is given in~\ref{dwell} when all the subsystems are stable, and point 4 in~\ref{app} when there are unstable subsystems as well. 
\end{rem}

\begin{exam}\label{gen_exam}
Let  the switched system~\ref{main} be defined on $\mathbb{R}^2$ and comprises of five subsystems  %%%%%%%%%%%%%%%%%%
\begin{eqnarray*}
A_1 =\left(
\begin{array}{c@{\enskip}c}
-2&0\\
0&-2
\end{array}\right), \
A_2 =\left(
\begin{array}{c@{\enskip}c}
-0.4&-0.03\\
1.4&0.04
\end{array}\right), \
A_3 = \left(
\begin{array}{c@{\enskip}c}
1&1\\
0&1
\end{array}\right),
\end{eqnarray*}
\begin{eqnarray*}
A_4 = \left(
\begin{array}{c@{\enskip}c}
0.1&0\\
0.1&0.2
\end{array}\right),\
A_5 = \left(
\begin{array}{c@{\enskip}c}
2&0.1\\0.1&2
\end{array}\right).
\end{eqnarray*}

Here $A_1, A_2$ are stable, and $A_3, A_4, A_5$ are unstable matrices. Let the underlying graph $\G$ be as in~\ref{adj_exam4}. The following linear topological sorting is considered for $\G_u\cup \G_b^2$: $v_3\ra v_5\ra v_4\ra v_1\ra v_2$, where $v_3$ and $v_5$ have no incoming edges. 

\noindent As per the notation in~\ref{acyclic}, $\rho=12.6987$. For $\zeta=0.99$, $\tau=12.8269$, $\rho'=0.603772$ and $\alpha'=232.32$ ($Q_4=\tau P_4$, $Q_1=\tau^2 P_1$, $Q_2=\tau^2 P_2$, $Q_3=P_3$, and $Q_5=P_5$). By Proposition~\ref{thmsufinal}, the switched system in $\mathcal{S}_{\G}(\tau,\eta)$ is stable if $\eta>0, \tau>0$ satisfy
\[
\tau>54.4812, \eta<1.21981.
\]
%Note that the bounds on $\tau$ and $\eta$ are inversely proportional, which is as expected since the system has to spend enough time in a stable state to offset the instability due to being in an unstable state. 
\end{exam}

\subsubsection{Commuting subsystem matrices}
It is well-known that if the matrices $A_1,\cdots,A_k$ are all Hurwitz and pair-wise commute with each other, then for any given graph $\G$, the switched system~\ref{main} is stable, for all $\G$-admissible switching signals $\sigma$, see~\cite{Lib} for proofs. When some subsystems are unstable, we can obtain concrete results when $\sigma$ belongs to the collections $\mathcal{S}_\G(\tau,\eta)$ and $\mathcal{S}_\G(\tau_1,\eta_1,\cdots,\tau_p,\eta_p)$. Since any two commuting diagonalizable matrices are simultaneously diagonalizable, if $A_1,\cdots,A_k$ pair-wise commute with each other, there exists an invertible matrix $P$ which simultaneously diagonalizes $A_1,\cdots,A_k$. Taking $P_i=P$, for all $i\in\textbf{k}$, we can further simplify the results (note that the columns of $P$ may not have unit norm). Thus, the stability conditions obtained in~\ref{stable-unstable} will be independent of the eigenvector matrices, since $\Vert P_j^{-1}P_i\Vert=\Vert I\Vert=1$, for all $(i,j)\in\mathcal{E}(\G)$. Further, the results can be improved as illustrated in~\ref{unidir}. 
%The condition~\ref{su} in~\ref{thmsu} becomes
%\[
%\limsup_{t\ra\infty} \dfrac{N_s(0,t)}{N_u(0,t)} > \dfrac{\lambda\tau}{\mu\eta},
%\] 
%~(\ref{bip1}) in~\ref{bipthmnew1} becomes
%\[
%\dfrac{\tau}{\eta}>\min_{i\in\mathbf{p}}\left(\dfrac{\mu_{s_i}}{\lambda_{s_i}}\right),
%\] 
%and the condition~(\ref{bip2}) in~\ref{bipthmnew2} becomes
%\[
%\dfrac{\tau_i}{\eta_i}>\dfrac{\mu_{s_i}}{\lambda_{s_i}},\ i=1,\cdots,p.
%\]  
%

\section{Concluding Remarks}
\noindent We have obtained stability conditions for the switched system (with all stable subsystems) using the concepts of standard decomposition and simple loop dwell time. Our results provide a mechanism of slow-fast (non-uniform) switching for stability of the switched system. Further, we considered the stability problem for switched systems which have both stable and unstable subsystems. We obtain conditions on the switching pattern, and bounds on the dwell time and the flee time to ensure stability of the switched system. Moreover, similar to switched systems with only stable subsystems, we obtain a slow-fast switching mechanism to stabilize the system.   

In this study, the switching sequence is deterministic. However, the concepts developed in this paper can be used to study almost sure stability of the switched system when the switching is stochastic, see~\cite{PS}. Also, one can explore the applicability of these results to large scale systems, and compare the computational costs of using these results of stability with existing results/models. Moreover, all the proofs presented here heavily rely on the choice of matrix spectral norm. It will be interesting to study the possibility of defining a new norm which gives tighter bounds, along the lines of~\cite{PRS}.

\section{Acknowledgements}
\noindent This work has been funded by SERB (DST), India.

\bibliographystyle{plain}
\bibliography{mybibfile_rev}

\begin{thebibliography}{10}

\bibitem{B1}
Edited by~I.~Belykh, M.~di~Bernardo, J.~Kurths, and M.~Porfiri.
\newblock Evolving dynamical networks.
\newblock {\em Physica D: Nonlinear Phenomena}, 267:1--132, 2014.

\bibitem{Cormen}
T.~H. Cormen, C.~E. Leiserson, R.~L. Rivest, and C.~Stein.
\newblock {\em Introduction to Algorithms}.
\newblock MIT Press and McGraw-Hill, 2001.

\bibitem{GJBP}
O.~Golovneva, R.~Jeter, I.~Belykh, and M.~Porfiri.
\newblock Windows of opportunity for synchronization in stochastically coupled
  maps.
\newblock {\em Physica D}, 340:1--13, 2017.

\bibitem{B4}
M.~Hasler, V.~Belykh, and I.~Belykh.
\newblock Blinking model and synchronization in small-world networks with a
  time-varying coupling.
\newblock {\em Physica D}, 195:188--206, 2004.

\bibitem{B3}
M.~Hasler, V.~Belykh, and I.~Belykh.
\newblock Dynamics of stochastically blinking systems. part i: Finite time
  properties.
\newblock {\em SIAM Journal on Applied Dynamical Systems}, 12:1007--1030, 2013.

\bibitem{B2}
M.~Hasler, V.~Belykh, and I.~Belykh.
\newblock Dynamics of stochastically blinking systems. part ii: Asymptotic
  properties.
\newblock {\em SIAM Journal on Applied Dynamical Systems}, 12:1031--1084, 2013.

\bibitem{HM}
J.P. Hespanha and A.S. Morse.
\newblock Stability of switched systems with average dwelltime.
\newblock {\em Proc. of the 38th Conf. on Decision and Control, Phoenix,
  Arizona}, pages 2655–--2660, 1999.

\bibitem{HXMA}
B.~Hu, X.~Xu, A.~N. Michel, and P.~J. Antsaklis.
\newblock Stability analysis for a class of nonlinear switched systems.
\newblock {\em Proceedings of the 38th IEEE Conference on Decision and
  Control}, pages 4374--4379, 1999.

\bibitem{IK}
F.~Ilhan and O.~Karabacak.
\newblock Graph-based dwell time computation methods for discrete-time switched
  linear systems.
\newblock {\em Asian Journal of Control}, 18:2018–--2026, 2016.

\bibitem{JB}
R.~Jeter and I.~Belykh.
\newblock Synchronization in on-off stochastic networks: windows of
  opportunity.
\newblock {\em IEEE Trans. Circuits Syst. I}, 62 (5):1260–--1269, 2015.

\bibitem{O1}
O.~Karabacak.
\newblock Dwell time and average dwell time methods based on the cycle ratio of
  the switching graph.
\newblock {\em Systems \& Control Letters}, 62:1032--1037, 2013.

\bibitem{KIO}
O.~Karabacak, F.~Ilhan, and I.~Oner.
\newblock Explicit sufficient stability conditions ¨ on dwell time of linear
  switched systems.
\newblock {\em Proc. of the 53rd Conf. on Decision and Control}, 2014.

\bibitem{KS}
O.~Karabacak and N.~S. Sengor.
\newblock A dwell time approach to the stability of switched linear systems
  based on the distance between eigenvector sets.
\newblock {\em International J. Sys. Sci.}, 40 (8):845--853, 2009.

\bibitem{Lib}
D.~Liberzon.
\newblock {\em Switching in systems and control}.
\newblock Birkhauser Boston Inc., Boston, MA, 2003.

\bibitem{Man}
J.~L. Mancilla-Aguilar, R.~Garcia, E.~Sontag, and Y.~Wang.
\newblock Uniform stability properties of switched systems with switchings
  governed by digraphs.
\newblock {\em Nonlinear Anal. Theor.}, 63:472--490, 2005.

\bibitem{Morse1}
S.~Morse.
\newblock Supervisory control of families of linear set-point
  controllers—part 1: exact matching.
\newblock {\em IEEE Trans. Automat. Contr.}, 41:1413--1431, 1996.

\bibitem{PRS}
M.~Porfiri, D.~G. Roberson, and D.~J. Stilwell.
\newblock Fast switching analysis of linear switched systems using exponential
  splitting.
\newblock {\em SIAM Journal on Control and Optimization}, 47 (5):2582--2597,
  2008.

\bibitem{PS}
M.~Porfiri and D.~J. Stilwell.
\newblock Consensus seeking over random weighted directed graphs.
\newblock {\em IEEE Transactions on Automatic Control}, 52 (9):1767--1773,
  2007.

\bibitem{PSB}
M.~Porfiri, D.~J. Stilwell, and E.~M. Bollt.
\newblock Synchronization in random weighted directed networks.
\newblock {\em IEEE Transactions on Circuits and Systems I}, 55
  (10):3170--3177, 2008.

\bibitem{PSBS}
M.~Porfiri, D.~J. Stilwell, E.~M. Bollt, and J.~D. Skufca.
\newblock Random talk: Random walk and synchronizability in a moving
  neighborhood network.
\newblock {\em Physica D Special Issue on Dynamics on Complex Networks},
  224:102--113, 2006.

\bibitem{SB}
J.~D. Skufca and E.~M. Bollt.
\newblock Communication and synchronization in disconnected networks with
  dynamic topology: Moving neighborhood networks.
\newblock {\em Mathematical Biosciences and Engineering}, 1 (2):347--359, 2004.

\bibitem{St}
D.~J. Stilwell, E.~M. Bolt, and D.~G. Robertson.
\newblock Sufficient conditions for fast switching synchronization in time
  varying network topologies.
\newblock {\em SIAM J. Dynamical Systems}, 5:140--156, 2006.

\bibitem{SU}
G.~Zhai, B.~Hu, K.~Yasuda, and A.~N. Michel.
\newblock Stability analysis of switched systems with stable and unstable
  subsystems: An average dwell time approach.
\newblock {\em International Journal of Systems Science}, 32:1055--1061, 2001.

\bibitem{ZWXS}
Ji-Shi Zhang, Yan-Wu Wang, Jiang-Wen Xiao, and Yan-Jun Shen.
\newblock Stability analysis of switched positive linear systems with stable
  and unstable subsystems.
\newblock {\em International Journal of Systems Science}, 45(12):2458--2465,
  2014.

\end{thebibliography}
\end{document}